\documentclass[12pt]{article}

\usepackage[utf8]{inputenc} 
\usepackage{textcomp} 
\usepackage{graphicx,epsfig}  
\usepackage{epstopdf} 
\usepackage{flafter}  
\usepackage[numbers]{natbib} 

\usepackage{amsmath,amssymb,amsthm}
\usepackage{bm}  
\usepackage{enumerate}
\usepackage[pdftex,bookmarks,colorlinks,breaklinks]{hyperref}  
\usepackage{setspace}

\hypersetup{linkcolor=red,citecolor=blue,filecolor=dullmagenta,urlcolor=darkblue} 

\newcommand{\htwo}{{ {\mathcal H}_2} }

\newcommand{\boldA}{\ensuremath{\boldsymbol A}}
\newcommand{\boldAr}{\ensuremath{\boldsymbol{A}_r}}
\newcommand{\boldb}{\ensuremath{\boldsymbol b}}
\newcommand{\boldbr}{\ensuremath{\boldsymbol {b}_r}}
\newcommand{\boldB}{\ensuremath{\boldsymbol B}}
\newcommand{\boldY}{\ensuremath{\boldsymbol Y}}

\newcommand{\boldc}{\ensuremath{\boldsymbol c}}
\newcommand{\boldcr}{\ensuremath{\boldsymbol{c}_r}}

\newcommand{\boldW}{\ensuremath{\boldsymbol{W}}}

\newcommand{\boldV}{\ensuremath{\boldsymbol{V}}}
\newcommand{\boldVr}{\ensuremath{\boldsymbol{V}_r}}
\newcommand{\reals}{\ensuremath{\mathbb{R}}}
\newcommand{\complex}{\ensuremath{\mathbb{C}}}

\newcommand{\Qr}{\ensuremath{\boldsymbol{Q}_r}}
\newcommand{\Ur}{\ensuremath{\boldsymbol{U}_r}}

\newcommand{\boldI}{\ensuremath{\boldsymbol{I}}}

\newcommand{\controlgram}{\ensuremath{\boldsymbol{\mathcal{P}}}}
\newcommand{\observegram}{\ensuremath{\boldsymbol{\mathcal{Q}}}}
\newcommand{\crosgram}{\ensuremath{\boldsymbol{\mathcal{X}}}}
\newcommand{\poles}{\ensuremath{\boldsymbol{\sigma}}}
\newcommand{\projX}{\ensuremath{\tilde{\boldsymbol{X}}_r}}
\newcommand{\fullX}{\ensuremath{\boldsymbol{X}}}
\newcommand{\boldx}{\ensuremath{\boldsymbol{x}}}

\newtheorem{thm}{Theorem}
\newtheorem{lemma}{Lemma}
\newtheorem{remark}{Remark}
\title{On the ADI method for the Sylvester Equation and the optimal-$\mathcal{H}_2$ points}
\author{Garret M. Flagg ~~and~~ Serkan Gugercin \\ \\
Department of Mathematics, Virginia Tech. \\
Blacksburg, VA, 24061-0123, USA\\
\{garretf,gugercin\}@vt.edu
}
\date{\small Preprint submitted to  \it{Applied Numerical Mathematics}, January 23, 2012}
\begin{document}
\maketitle
\abstract
The ADI iteration is closely related to the rational Krylov projection methods for constructing low rank approximations to the solution of Sylvester equation.  In this paper we show that the ADI and rational Krylov approximations are in fact equivalent when a special choice of shifts are employed in both methods. We will call these shifts pseudo $\mathcal{H}_2$-optimal shifts.  These shifts are also optimal in the sense that for the Lyapunov equation, they yield a residual which is orthogonal to the rational Krylov projection subspace.  Via several examples, we show that  the pseudo $\mathcal{H}_2$-optimal shifts consistently yield nearly optimal low rank approximations to the solutions of the Lyapunov equations.

\section{Introduction}
Let $\boldA \in \reals^{n\times n}$, $\boldB \in \reals^{m\times m}$ and $\boldY \in \reals^{n \times m}$ be 
given matrices. Then, the Sylvester equation for the unknown matrix $\fullX \in \reals^{n\times m}$ is given by
\begin{equation}\label{sylv_equation}
\boldA\fullX +\fullX\boldB+ \boldY = \mathbf{0}.
\end{equation}
The equation (\ref{sylv_equation}) has a unique solution if and only if $\lambda_i(\boldA) + \lambda_j(\boldB) \neq 0$ for 
$i=1,\ldots,n$ and $j=1,\ldots,m$.
 A special case of the Sylvester equation is the Lyapunov equation, where $\boldB=\boldA^*$ and $\boldY = \boldY^*\ge \boldsymbol{0}$.  Both the Sylvester and Lyapunov equations are an important tool in the analysis of asymptotically stable linear dynamical systems of the form 
\begin{eqnarray}
\label{ltisystemintro}
\dot{\boldx}(t)  =  \boldA \boldx(t) + \boldb\,u(t),~~~
y(t)  =  \boldc^* \boldx(t),
\end{eqnarray}
where $\boldA \in \reals^{n \times n}$ and $\boldb,\boldc^* \in \reals^{n}$.
In (\ref{ltisystemintro}),
$\boldx(t)\in \reals^n$, $u(t)\in \reals$, $y(t)\in \reals$, are, respectively the
\emph{state}, \emph{input}, and \emph{output},
of the underlying system. While the cross gramian 
$\crosgram$ of (\ref{ltisystemintro}) solves the Sylvester equation 
\begin{align} \label{cross}
\boldA \crosgram + \crosgram\boldA  + \boldb\boldc^* =  \mathbf{0}, 
\end{align}
 the controllability 
gramian $\controlgram$ and the observability gramian $\observegram$ solve the  Lyapunov equations 
\begin{align} \label{reaobs}
\boldA\controlgram +\controlgram\boldA^* + \boldb\boldb^* =  \mathbf{0} \qquad {\rm and}\qquad
\boldA^*\observegram+\observegram\boldA + \boldc^*\boldc=  \mathbf{0},
\end{align}
respectively. These three gramians are of fundamental importance especially in the concept of model reduction, see \cite{AntB}.  In what follows, we will mainly
focus on the Sylvester equation (\ref{sylv_equation}) where $\boldY$ is rank-$1$; hence our discussion already contains the Lyapunov equations as a special case.

The standard direct method for solving (\ref{sylv_equation}) is due to Bartels and Stewart 
\cite{bartels1972solution}. However, this method requires dense matrix operations such as the Schur decomposition; thus is not applicable in large-scale settings. For large-scale settings, iterative methods have been developed that take advantage of the sparsity and the low-rank structure of $\boldY$. The two most common ones are the Alternating Direction Implicit (ADI) Method (\cite{penzl200clr,benner2003statespace,Benner20091035, li2004low,gugercin2003amodified,stykel2004gbm,sorensen2002thesylvester,heinkenschloss2008btm,sabino2007solution,wachspress1988adi,truhar2007adi,wachspress2008trail}) and the (rational) Krylov projection methods (\cite{hu1992krylov,jbilou2006low,druskin:1875, jaimoukha1994krylov,simoncini2008new,stykel2011krylov,el2002block,bao2007new}).

The ADI method was first introduced by Peaceman and Rachford \cite{peaceman1955numerical} for solving parabolic and elliptic PDEs, and was later adapted to solving the Sylvester equation by Wachspress in \cite{wachspress1988adi}.  It is a fixed point iteration scheme for approximating \fullX. 
Given two sequences of shifts $\{\alpha_1, \alpha_2, \dots, \alpha_r, \dots\}, \{\beta_1,\beta_2,\dots,\beta_r,\dots\} \subset \complex$ and an initial guess $\fullX_0$, the ADI iteration for (\ref{sylv_equation}) proceeds as follows :
\begin{align}    \label{adim}
\fullX_i=&(\boldA-\alpha_i\boldI)(\boldA+\beta_i\boldI)^{-1}\fullX_{i-1}(\boldB-\beta_i\boldI)(\boldB+\alpha_i\boldI)^{-1}\\
&-(\alpha_i+\beta_i)(\boldA+\beta_i\boldI)^{-1}\boldY(\boldB+\alpha_i\boldI)^{-1}.
\end{align}

The performance of the ADI iteration depends heavily on the choice of shifts used in the iteration.  Several schemes have been developed for making asymptotically optimal shift selections if some information is known about the boundaries of the numerical range of \boldA, and $\boldB$. See \cite{wachspress1963extended,wachspress1988adi,starke1991optimal,Starke,sabino2007solution} and the references therein for further details on the shift selection problem in the ADI iteration.  

A closely related method to the ADI iteration is the rational Krylov projection method (RKPM).  In the RKPM, the Sylvester equation $\boldA\fullX +\fullX\boldB+ \boldY = \mathbf{0}$ is projected onto the rational Krylov subspaces $\mathcal{K}^{\text{rat}}_{r}(\boldA,\boldb,\boldsymbol{\sigma})=\text{span}\{(\sigma_1\boldI-\boldA)^{-1}\boldb, \dots, (\sigma_r\boldI-\boldA)^{-1}\boldb\}$ 
and $\mathcal{K}^{\text{rat}}_{r}(\boldB^*,\boldc,\bar{\boldsymbol{\mu}})$
 where  \poles$=\{\sigma_1,\dots \sigma_{r}\}$, and $\bar{\boldsymbol{\mu}}=\{\bar\mu_1,\dots, \bar \mu_r\}$ are the sets of shifts used to construct the respective  rational Krylov spaces and $\bar \nu$ denotes the conjugate of $\nu$.   
 See \cite{beckermann2010convergence} for further details regarding $\mathcal{K}^{\text{rat}}_{r}(\boldA,\boldb,\boldsymbol{\sigma})$, and constructing an orthonormal basis via the rational Arnoldi iteration.  Let $\Qr$ and $\Ur$ denote the orthonormal basis for 
 $\mathcal{K}^{\text{rat}}_{r}(\boldA,\boldb,\boldsymbol{\sigma})$ and 
$\mathcal{K}^{\text{rat}}_{r}(\boldB^*,\boldc,\bar{\boldsymbol{\mu}})$. Then,
the RKPM approximation is constructed by first solving
\begin{equation} \label{proj_sylvester}
\Qr^*\boldA\Qr\projX+\projX\Ur^*\boldB^*\Ur+ \Qr^*\boldb\boldc^*\Ur =\boldsymbol{0}
\end{equation}
and then approximating $\fullX$ by $\Qr\projX\Ur^*$.  The solution of the projected Sylvester equation (\ref{proj_sylvester}) is very cheap.  Like the ADI method, the RKPM method also relies heavily on a good choice of shifts to produce accurate results. In the next section we will derive results that show for a certain choice of shifts, the RKPM and ADI methods are indeed equivalent.  

Since in almost all applications, the quantities $\boldsymbol{A}$, $\boldsymbol{B}$, $\boldsymbol{b}$, and $\boldsymbol{c}$ are real, we will assume that the set of shifts  
$\boldsymbol{\sigma}$ and $\boldsymbol{\mu}$ are closed under conjugation so that the approximants are real as well. This will guarantee that the orthonormal bases $\Qr$ 
for $\mathcal{K}^{\text{rat}}_{r}(\boldA,\boldb,\boldsymbol{\sigma})$
and $\Ur$ for $\mathcal{K}^{\text{rat}}_{r}(\boldB^*,\boldc,\bar{\boldsymbol{\mu}})$ can be computed to be real as well.
\section{Equivalence of the ADI and Rational Krylov Projection Methods for pseudo-$\htwo$ optimal points }
In this section, we present our main results illustrating the connection between the ADI and RKPM. Since the discussion requires the concept of $\htwo$-optimal points for model reduction, we first briefly review the $\htwo$ approximation problem. 

\subsection{Optimal $\htwo$ model reduction}
For a full-order model as given in (\ref{ltisystemintro}),
the model reduction problem seeks to construct a dynamical system
\begin{eqnarray} \label{redsysintro}
\dot{\boldx}_r (t) = \boldA_r \boldx_r (t) + \boldb_r u(t),~
y_r(t)  =  \boldc_r^* \boldx_r (t)  
\end{eqnarray}
of much smaller dimension $r \ll n$, with  $\boldA_r \in \reals^{r \times r}$ and
$\boldb_r,\,\boldc_r^* \in \reals^{r}$ such that $y_r(t)$ approximates $y(t)$ well for a wide range of inputs $u(t)$.
The reduced-model in (\ref{ltisystemintro})  is usually obtained  via state-space projection:
Two matrices $\boldV_r, \ \boldW_r \in \reals^{n \times r}$ are constructed with 
$\boldW_r^*\boldV_r = \boldI_r$
to produce
\begin{equation}  \label{eqn:reduction}
\boldA_r=\boldW_r^* \boldA\boldV_r,\qquad
\boldb_r=\boldW_r^*\boldb,\qquad{\rm and}\qquad
\boldc_r= \boldV_r \,\boldc
\end{equation}

 One can measure the quality of the approximation using 
the concept of transfer function. By taking the Laplace transforms of  (\ref{ltisystemintro}) and 
(\ref{redsysintro}), one obtains the transfer functions
 $H(s)=\boldc(s\boldI-\boldA)^{-1}\boldb$ and  $H_r(s)=\boldcr(s\boldI_r-\boldAr)^{-1}\boldbr$, respectively.  Hence, one can consider model reduction in terms of these transfer functions as
 approximating a degree-$n$ rational function $H(s)$ with a degree-$r$ one $H_r(s)$. For more details on model reduction of linear dynamical systems, see \cite{AntB}. 
  
In this paper, we focus on the $\htwo$-norm to measure accuracy of the reduced-model. 
The $\mathcal{H}_2$ optimal model reduction problem seeks to construct a reduced system as in (\ref{redsysintro}), so that $H_r(s)$ minimizes  the $\htwo$ error over all stable linear dynamical systems of the form (\ref{redsysintro}), i.e.
\begin{equation} \label{h2opt}
\|H -H_r \|_{ _{\mathcal{H}_2}} =  \min_{{\small \rm deg} (\tilde{H}_r) = r} \| H - \tilde{H}_r\|_\htwo
\end{equation}
where
$$\|H -H_r \|_{ _{\mathcal{H}_2}} = \left( \frac{1}{2\pi} \int_{-\infty}^\infty \left| H(\imath \omega) -  H_r(\imath \omega)  \right|^2 d\omega\right)^{1/2}.$$ 
 Several methods have been introduced to solve (\ref{h2opt}); see, for example,
\cite{spanos1992anewalgorithm,hyland1985theoptimal,wilson1970optimum,zigic1993contragredient,meieriii1967approximation,H2,gugercin2006rki,vandooren2008hom,bunse-gerstner2009hom,gugercin2005irk,beattie2007kbm,beattie2009trm}, and the references therein. Since the optimization problem (\ref{h2opt}) is nonconvex, the common approach involves finding  reduced-order models satisfying the first-order necessary conditions of 
 $\htwo$-optimality.  The next theorem states the interpolation-based necessary conditions for $\mathcal{H}_2$ optimality introduced by
 Meier and Luenberger \cite{meieriii1967approximation}.
\begin{thm} \label{thm:meier}\cite{meieriii1967approximation,H2}
Given a full-order system $H(s)$ of order $n$, if $H_r(s)=\sum\limits_{i=1}^r\frac{\phi_i}{s-{\lambda}_i}$ is an $\mathcal{H}_2$ optimal approximation to $H(s)$, then
\begin{eqnarray}
H(-{\lambda}_i) &=&H_r(-{\lambda_i})~~ \mbox{for}\qquad i=1, \dots, r,~~\mbox{and} \label{HHr}\\ 
H'(-{\lambda}_i)&=&H_r'(-{\lambda_i})~~ \mbox{for}\qquad i=1, \dots, r  \label{HpHrp}
\end{eqnarray}
\end{thm}

A reduced-order model which satisfies the $\htwo$-optimality conditions can be obtained by using the Iterative Rational Krylov Algorithm (IRKA) of Gugercin et. al. in \cite{H2}. However, in this paper we will focus on satisfying only (\ref{HHr}) (without the derivative condition). We will call these interpolation points \emph{pseudo $\htwo$-optimal points} to emphasize that they only satisfy a subset of the optimality conditions. 
In terms of the projection framework  (\ref{eqn:reduction}) for model reduction, this corresponds to finding interpolation points  $\boldsymbol{\sigma} = \{\sigma_1,\ldots,\sigma_r\}$ and choosing, in (\ref{eqn:reduction}),  $\boldV_r = \boldW_r = \Qr$ where $\Qr$ is an orthonormal basis for the rational Krylov subspace $\mathcal{K}^{\text{rat}}_{r}(\boldA,\boldb,\boldsymbol{\sigma})$ so that  
$\{\lambda_1,\ldots,\lambda_r\}$, i.e. 
the eigenvalues of $\boldA_r  = \Qr^T \boldA \Qr$, become the mirror images of the interpolations points $\boldsymbol{\sigma} =\{\sigma_1,\ldots,\sigma_r\} $, i.e.  
\begin{equation} \label{h2onesided}
\boldsymbol{\lambda}(\boldA_r) =\boldsymbol{\lambda}(\Qr^T \boldA \Qr) =   -\boldsymbol{\sigma}.
\end{equation}
The  emph{pseudo $\htwo$-optimal points} interpolation points can be computed iteratively in a manner similar to IRKA  \cite{H2}
as done in  \cite{gugercin2011phmimo} for port-Hamiltonian systems.

\subsection{The ADI Iteration and Rational Krylov Projection Method}
 
The main theorem requires the following lemma, which connects the ADI approximation for the Sylvester equation with rational Krylov subspaces. This extends an earlier result by Li and White \cite{li2002low} which establishes a similar connection for the  the case of the Lyapunov equation. 
\begin{lemma}\label{lemma:l1}
Let $\boldY=\boldb\boldc^*$, where $\boldb \in \reals^{n}$ and $\boldc \in \reals^{m}$. Let $\{\sigma_1,\dots, \sigma_r\}$ and $\{\mu_1,\dots, \mu_r\}$ be two collections of shifts that satisfy $\Re(\mu_i),\Re(\sigma_i)>0$ for $i=1,\dots, r$. Suppose $\boldsymbol{X}_r$ is the approximate solution to the Sylvester equation (\ref{sylv_equation}) obtained by applying the pair of shifts $\alpha_i=-\sigma_i$  and  $\beta_i=-\mu_i$  in  the ADI iteration (\ref{adim}) for 
$i=1,\ldots,r$ with $\boldsymbol{X}_0 = \mathbf{0}$.  Then  there exist $\boldsymbol{L}_r\in \complex^{n\times r}$ and $\boldsymbol{R}_r \in \complex^{m \times r}$ such that $\boldsymbol{X}_r=\boldsymbol{L}_r\boldsymbol{R}_r^*$ and $\text{colspan}(\boldsymbol{L}_r)\subset\mathcal{K}^{\text{rat}}_{r}(\boldA,\boldb,\boldsymbol{\mu})$ and $\text{colspan}(\boldsymbol{R}_r)\subset\mathcal{K}^{\text{rat}}_{r}(\boldB^*,\boldc,\bar{\boldsymbol{\sigma}})$
\end{lemma}
\begin{proof}
The proof is given by induction on $i$, the iteration step.  First note that for $i=1$, $\boldsymbol{X}_1=(\mu_1+\sigma_1)(\boldA-\mu_1\boldI)^{-1}\boldb\boldc^*(\boldB-\sigma_1\boldI)^{-1}$, so let $\boldsymbol{L}_1=[(\mu_1+\sigma_1)(\boldA-\mu_1\boldI)^{-1}\boldb]$ and $\boldsymbol{R}_1=[(\boldB^*-\bar{\sigma_1}\boldI)^{-1}\boldc]$.  Then $\boldsymbol{L}_1$ and $\boldsymbol{R}_1$ clearly satisfy the hypothesis and $\boldsymbol{X}_1 = \boldsymbol{L}_1\boldsymbol{R}_1^*$.  
Now suppose that  the statement holds for $\boldsymbol{X}_{i}$. Then, for $j=1,\dots, i$, the $j^{\rm th}$ column of $\boldsymbol{L}_i$ is $T_i^{(j)}(\boldA)\boldb$, where $T_i^{(j)}(\lambda)$ is a proper rational function that lies in the span of $\{\frac{1}{\lambda-\mu_1},\dots,\frac{1}{\lambda-\mu_i}\}$.  Similarly, the $j^{\rm th}$ column of $\boldsymbol{R}_i$ is $S_i^{(j)}(\boldB^*)\boldc$, where $S_i^{(j)}(\lambda)$ lies in the span of $\{\frac{1}{\lambda-\bar{\sigma_1}},\dots,\frac{1}{\lambda-\bar{\sigma_i}}\}$.  
 Therefore $\boldsymbol{X}_{i+1}$ can be written as 
\begin{align*}
\boldsymbol{X}_{i+1}=&(\boldA+\sigma_{i+1}\boldI)(\boldA-\mu_{i+1}\boldI)^{-1}\boldsymbol{L}\boldsymbol{R}^*(\boldB+\mu_{i+1}\boldI)(\boldB-\sigma_{i+1}\boldB)^{-1}\\
&+(\mu_{i+1}+\sigma_{i+1})(\boldA-\mu_{i+1}\boldI)^{-1}\boldb\boldc^*(\boldB-\sigma_{i+1}\boldI)^{-1}\\
=&\sum\limits_{j=1}^{i}(\boldA+\sigma_{i+1}\boldI)(\boldA-\mu_{i+1}\boldI)^{-1}T_i^{(j)}(\boldA)\boldb\boldc^* S_i^{(j)}(\boldB)(\boldB+\mu_{i+1}\boldI)(\boldB-\sigma_{i+1}\boldI)^{-1}\\
&+(\mu_{i+1}+\sigma_{i+1})(\boldA-\mu_{i+1}\boldI)^{-1}\boldb\boldc^*(\boldB-\sigma_{i+1}\boldI)^{-1}
\end{align*}
For $j=1,\dots i$, let the $j^{\rm th}$ column of $\boldsymbol{L}_{i+1}$ be $(\boldA+\sigma_{i+1}\boldI)(\boldA-\mu_{i+1}\boldI)^{-1}T_i^{(j)}(\boldA)\boldb$  and let the $(i+1)^{\rm th}$ column be $T_{i+1}^{(i+1)}(\boldA)\boldb=(\mu_{i+1}+\sigma_{i+1})(\boldA-\mu_{i+1}\boldI)^{-1}\boldb$.   Then clearly $\text{colspan}(\boldsymbol{L}_{i+1} )\subset \mathcal{K}^{\text{rat}}_{i+1}(\boldA,\boldb,\boldsymbol{\mu})$.  Similarly, let  $(\boldB^*-\bar{\sigma}_{i+1}\boldI)^{-1}(\boldB^*+\bar{\mu}_{i+1}\boldI)S_i^{(j)}(\boldB^*)\boldc$ be the $j^{\rm th}$ column of $\boldsymbol{R}_{i+1}$ for $j=1,\dots, i$, and  $S_{i+1}^{(i+1)}(\boldB^*)\boldc=(\boldB^*-\bar{\sigma}_{i+1}\boldI)^{-1}\boldc$ be the $(i+1)^{\rm th}$ column.  Then $\text{colspan}(\boldsymbol{R}_{i+1})\subset\mathcal{K}^{\text{rat}}_{i+1}(\boldB^*,\boldc,\bar{\boldsymbol{\sigma}})$. Finally, we note that by construction, $\boldsymbol{X}_{i+1}=\boldsymbol{L}_{i+1}\boldsymbol{R}_{i+1}^*$.
\end{proof}
Next, we give our first main result showing that  the approximate solution of 
the Sylvester equation (\ref{sylv_equation}) by ADI and RKPM are indeed equivalent 
when the shifts are chosen as pseudo-$\htwo$ optimal points.  This result applied to the special case of Lyapunov equation was first presented at the 2010 SIAM Annual Meeting \cite{Flagg_presentation} then later published independently in \cite{druskin:1875}.  
Our new result here, on the other hand, is more general than both \cite{Flagg_presentation} and 
\cite{druskin:1875} since it tackles the case of Sylvester equation and includes the Lyapunov equation as a special case. Moreover, while
the proof given in \cite{druskin:1875}  for the special case of Lyapunov equation
makes use of a novel connection between the ADI iteration and the so-called Skeleton approximation framework first developed in the work of Tyrtyshnikov \cite{tyrtyshnikov1996mosaic}, 
the proof we provide here for the more general Sylvester equation case is given directly in terms of  rational Krylov interpolation conditions, and in that sense is simpler. 
 \begin{thm} \label{sylvester_eq}
Given the Sylvester equation (\ref{sylv_equation}) with $\boldY=\boldb\boldc^*$, where $\boldb \in \reals^{n}$ and $\boldc \in \reals^{m}$, let $\Qr \in \reals^{n \times r}$ be an orthonormal basis for the rational Krylov subspace $\mathcal{K}^{\text{rat}}_{r}(\boldA,\boldb,\boldsymbol{\sigma})$ and let $\Ur \in \reals^{m \times r}$ be an orthonormal basis for the rational Krylov subspace
$\mathcal{K}^{\text{rat}}_{r}(\boldB^*,\boldc,\bar{\boldsymbol{\sigma}})$ 
for a set of shifts $\boldsymbol{\sigma} = \{\sigma_1,\ldots,\sigma_r\}$
 where $\Re(\sigma_i)>0$ for $i=1,\ldots,r$.  Let $\projX  \in \reals^{r \times r}$ solve the projected Sylvester equation
\begin{equation} \label{projsylvequ}
\Qr^*\boldA\Qr\projX+\projX\Ur^* \boldB\Ur+ \Qr^*\boldb\boldc^*\Ur =\boldsymbol{0},
\end{equation}
and let $\fullX_r \in \reals^{n \times m}$ be computed by applying the shifts $\alpha_i=-\sigma_i$ and $\beta_i=-\sigma_i$ to exactly $r$ steps of the ADI iteration (\ref{adim}) for $i=1,\ldots,r$.  Then $\fullX_r= \Qr \projX \Ur^*$ if and only if either $\boldsymbol{\lambda}(\Qr^*\boldA\Qr)=-\{\sigma_1,\dots,\sigma_r\}$ or $\boldsymbol{\lambda}(\Ur^*\boldB\Ur)=-\{\sigma_1,\dots,\sigma_r\}$.
\end{thm}
\begin{proof}
($\Leftarrow$) First suppose that $\lambda(\Qr^*\boldA\Qr)=-\{\sigma_1,\dots,\sigma_r\}$.  The proof remains the same if we instead suppose that $\lambda(\Ur^*\boldB\Ur)=-\{\sigma_1,\dots,\sigma_r\}$.  Let $\boldAr=\Qr^*\boldA\Qr$, and $\boldbr=\Qr^*\boldb$, $\boldB_r=\Ur^*\boldB\Ur$, and $\boldcr=\Ur^*\boldc$.  Note that after we apply
$r$ steps of the  ADI iteration with the set of shifts  $\alpha_i = \beta_i = -\sigma_i$ to  the projected Sylvester equation (\ref{projsylvequ}), we obtain the exact solution $\projX$, since $\boldsymbol{\lambda}(\boldAr)=-\{\sigma_1,\dots,\sigma_r\}$.  By Lemma \ref{lemma:l1}, at the $rth$ step of the ADI iteration $\projX=\tilde{\boldsymbol{L}}_r\tilde{\boldsymbol{R}}_r^*$ where $\tilde{\boldsymbol{L}}_r=[T^{(1)}(\boldAr)\boldbr, \dots, T^{(r)}(\boldAr)\boldbr ]$ where $T^{(i)}(\boldAr)\boldbr$ are rational functions that lie in $\mathcal{K}^{\text{rat}}_{r}(\boldAr,\boldbr,\boldsymbol{\sigma}).$ Similarly $\tilde{\boldsymbol{R}}_r=[S^{(1)}(\boldB_r^*)\boldcr^*, \dots, S^{(r)}(\boldB_r^*)\boldcr^*]$ where the $S^{(i)}(\boldB_r^*)\boldcr$ are rational functions that lie in  $\mathcal{K}^{\text{rat}}_{r}(\boldB_r^*,\boldcr^*,\bar{\boldsymbol{\sigma}}).$ Furthermore, for the same shifts, $\alpha_i = \beta_i = -\sigma_i$ for $i=1,\ldots,r$, applied to $r$ steps of the ADI iteration on the full Sylvester equation (\ref{sylv_equation}), we have $\fullX_r= \boldsymbol{L}_r\boldsymbol{R}_r^*$ and $\boldsymbol{L}_r= [T^{(1)}(\boldA)\boldb, \dots, T^{(r)}(\boldA)\boldb ]$ and $\boldsymbol{R}_r=[S^{(1)}(\boldB^*)\boldc, \dots, S^{(r)}(\boldB^*)\boldc ]$. Thus it is sufficient to show that $\Qr \tilde{\boldsymbol{L}}_r=\boldsymbol{L}_r$ and that $\Ur \tilde{\boldsymbol{R}}_r=\boldsymbol{R}_r$.  Without loss of generality consider just the former equation.  This, in turn, amounts to showing that $\Qr T^{(i)}(\boldAr)\boldbr=T^{(i)}(\boldA)\boldb$.  If $T_i(\boldA)\boldb$ are a set of orthogonal rational functions that span  $\mathcal{K}^{\text{rat}}_{r}(\boldA,\boldb,\boldsymbol{\sigma})$, then it is sufficient to show that 
\begin{equation}\label{finalshow}
\Qr T_i(\boldAr)\boldbr=T_i(\boldA)\boldb.
\end{equation}
Equality (\ref{finalshow}) follows readily from the interpolation properties of the Galerkin projection, which we show below.   First, note that due to the interpolation properties of the Galerkin projection, $\Qr(\sigma_i\boldI_r-\boldAr)^{-1}\boldbr=(\sigma_i\boldI-\boldA)^{-1}\boldb$.
Let $\boldVr=[(\sigma_1\boldI-\boldA)^{-1}\boldb \dots (\sigma_r\boldI-\boldA)^{-1}\boldb]$. Then,  for some $\boldsymbol{x}\in \reals^{r}$, 
\begin{align}
\boldVr\boldsymbol{x}=T_i(\boldA)\boldb
=\Qr[(\sigma_1\boldI_r-\boldAr)^{-1}\boldbr \dots (\sigma_r\boldI_r-\boldAr)^{-1}\boldbr]\boldsymbol{x}
=\Qr T_i(\boldAr)\boldbr,
\end{align}
which proves (\ref{finalshow}).

($\Rightarrow$) Let $\tilde{\boldsymbol{X}}_r$ be the solution of 
\begin{equation} \label{small_lyapunov}
\Qr^*\boldA\Qr\projX+\projX\Ur^* \boldB\Ur+ \Qr^*\boldb\boldc^*\Ur =\boldsymbol{0}
\end{equation}
where \Qr\ is an orthonormal basis for $\mathcal{K}^{\text{rat}}_{r}(\boldA, \boldb,\boldsymbol{\sigma})$ and \Ur\ is an orthonormal basis for $\mathcal{K}^{\text{rat}}_{r}(\boldB^*,\boldc,\bar{\boldsymbol{\sigma}})$. Suppose that $\Qr \tilde{\boldsymbol{X}}_r \Ur^*=\boldsymbol{X}_r$.    Let $\hat{\boldsymbol{X}}_r$ be the approximate solution of (\ref{small_lyapunov}) resulting from applying the shifts $\alpha_i = \beta_i = -\sigma_i$ for $i=1,\ldots,r$ to exactly $r$ steps of the ADI iteration (\ref{adim}).  By the interpolation result given in the proof above,  $\Qr \hat{\boldsymbol{X}}_r \Ur^*=\boldsymbol{X}_r$.  It follows from the assumptions that, $\Qr \hat{\boldsymbol{X}}_r \Ur^*=\Qr \tilde{\boldsymbol{X}}_r \Ur^*$, so $\hat{\boldsymbol{X}}_r=\tilde{\boldsymbol{X}}_r$.  But this means that $\hat{\boldsymbol{X}}_r$ solves (\ref{small_lyapunov}), and so either $\lambda(\Qr^*\boldA\Qr)=-\{\sigma_1,\dots,\sigma_r\}$ or $\lambda(\Ur^*\boldB\Ur)=-\{\sigma_1,\dots,\sigma_r\}$.
\end{proof}
\begin{remark}
This theorem shows that the ADI approximation for the Sylvester equation is equivalent to lifting 
the solution of the projected Sylvester equation back to the original dimension when
either
$ \boldsymbol{\lambda}(\Qr^*\boldA\Qr)=-\{\sigma_1,\dots,\sigma_r\}$ or $ \boldsymbol{\lambda}(\Ur^*\boldB\Ur)=-\{\sigma_1,\dots,\sigma_r\}$; hence the two most common approximation methods for solving a Sylvester equation is indeed equivalent for these special shift selection. Recalling pseudeo-$\htwo$ optimality condition (\ref{h2onesided}), for a given $r$, these special shifts are indeed exactly the pseudo-$\htwo$ optimal shifts for a dynamical system $H_1(s) = \boldsymbol{z}_1(s \boldI - \boldA)^{-1}\boldb$ or  
$H_2(s) = \boldsymbol{z}_2(s \boldI - \boldB^*)^{-1}\boldc^*$ where $\boldsymbol{z}_1$ and $\boldsymbol{z}_2$
are vectors of appropriate sizes.
\end{remark}

\subsection{Orthogonality in the case of Lyapunov equation} 
The parameters for which the ADI iteration and the rational Krylov projections coincide also satisfy orthogonality conditions on the residual for the special case of the Lyapunov equation  
\begin{equation} \label{lyapbb}
\boldA \boldsymbol{X}+\boldsymbol{X}\boldA^*+\boldb\boldb^*=\boldsymbol{0}
\end{equation} For a given approximation $\fullX_r$ to the solution $\fullX$, define the residual $\boldsymbol{R}$ as
\begin{equation}  \label{lyapres}
\boldsymbol{R}=\boldA \boldsymbol{X}_r+\boldsymbol{X}_r\boldA^*+\boldb\boldb^*.
\end{equation}
The following result was first given in \cite{druskin:1875}.  Here we present a new and more concise proof of the orthogonality result in terms of the special interpolation properties of the pseudo $\htwo$-optimal shifts.  
\begin{thm}  \label{thm:lyap}
Given $\boldA \boldsymbol{X}+\boldsymbol{X}\boldA^*+\boldb\boldb^*= \boldsymbol{0}$, let  $\tilde{\boldsymbol{X}}_r \in \reals^{r \times r}$
solve the projected Lyapunov equation 
$$
\Qr^*\boldA\Qr \projX  + \projX \Qr^* \boldA\Qr+ \Qr^*\boldb\boldb^*\Qr =\boldsymbol{0},
$$
where  \Qr\ is an orthonormal basis for the $\mathcal{K}^{\text{rat}}_{r}(\boldA,\boldb,\boldsymbol{\sigma})$ 
with $\boldsymbol{\sigma} = \{\sigma_1,\dots,\sigma_r\}$
Let $\fullX_r=\Qr\tilde{\boldsymbol{X}}_r\Qr^*$.Then $\Qr^*\boldsymbol{R}=\boldsymbol{0}$ if and only if $\boldsymbol{\lambda}(\Qr^*\boldA\Qr)=-\{{\sigma}_1,\dots,{\sigma}_r\}$ where $\boldsymbol{R}$ is the residual defined in (\ref{lyapres}).
\end{thm}

\begin{proof}
($\Rightarrow$)  Suppose that $\Qr^*\boldsymbol{R}=\boldsymbol{0}$.  Multiplying $(\ref{lyapres})$ with $\Qr^*$ from the
left and  then transposing the resulting equation leads to
\begin{align}  \label{Aqr1}
\boldA\Qr\projX+\Qr\projX \Qr^*\boldA^*\Qr+\boldb\boldb^*\Qr&=\boldsymbol{0}.
\end{align}
Let $\boldAr=\Qr^*\boldA\Qr = \boldsymbol{T} \boldsymbol{\Lambda} \boldsymbol{T}^{-1}$
be the eigenvalue decomposition of $\boldAr$ where 
$\boldsymbol{\Lambda}= {\rm diag}({\lambda}_1,\ldots, {\lambda}_r)$.  Plug these expressions into  (\ref{Aqr1}), and right multiply by $\boldsymbol{T}^{-*}$ to obtain
\begin{align}
\Qr\projX\boldsymbol{T}^{-*} \boldsymbol{\Lambda}^*+\boldA\Qr\projX\boldsymbol{T}^{-*}+\boldb\boldb^*\Qr\boldsymbol{T}^{-*}&=\boldsymbol{0} \label{interpolation_sylvester}
\end{align}
  Let $\zeta_i$ be the $i^{\rm th}$ entry of $\boldb^*\Qr\boldsymbol{T}^{-*}$.  Then it is straightforward to show that the $i^{\rm th}$ column of $\Qr\projX\boldsymbol{T}^{-*}$ must be $(-\bar{\lambda}_i\boldI-\boldA)^{-1}\boldb\zeta_i$. Thus, it follows that $\mathcal{K}^{\text{rat}}_{r}(\boldA,\boldb,\boldsymbol{\sigma})=\mathcal{K}^{\text{rat}}_{r}(\boldA,\boldb,-\bar{\boldsymbol{\lambda}})$, where  ${\boldsymbol{\lambda}}=\{{\lambda}_1, \dots,{\lambda}_r\}$.  Since both sets $\boldsymbol{\sigma}$ and $\boldsymbol{\lambda}$ are closed under conjugation, after an appropriate reordering, we obtain $\sigma_i=-{\lambda}_i$.\\
($\Leftarrow$) Observe that
\begin{align}
\boldAr\projX+\projX\boldAr^*+\Qr^*\boldb\boldb^*\Qr&=\boldsymbol{0} \Rightarrow\\
\boldAr\projX\boldsymbol{T}^{-*}+\projX\boldsymbol{T}^{-*} \boldsymbol{\Lambda}^*+\Qr^*\boldb\boldb^*\Qr\boldsymbol{T}^{-*}&=\boldsymbol{0}.
\end{align}

Thus, the $ith$ column of $\projX\boldsymbol{T}^{-*}$ is $(-\bar{\lambda}_i\boldI_r-\boldAr)^{-1}\Qr^*\boldb\zeta_i$.  But since \Qr\ is an orthonormal basis for $\mathcal{K}^{\text{rat}}_{r}(\boldA,\tilde{\boldb},\boldsymbol{\sigma})$, and $\lambda_i=-\bar{\sigma}_i$, this means 
\begin{equation}
\Qr(-\bar{\lambda}_i\boldI_r-\boldAr)^{-1}\Qr^*\boldb\zeta_i=(-\bar{\lambda}_i\boldI-\boldA)^{-1}\boldb\zeta_i=(\Qr\projX\boldsymbol{T}^{-*}) \boldsymbol{e}_i,
\end{equation}
where $\boldsymbol{e}_i$ is the $ith$ unit vector.
Thus,
\begin{align}
\Qr\projX\boldsymbol{T}^{-*}\boldsymbol{\Lambda}^* + \boldA\Qr\projX\boldsymbol{T}^{-*} +\boldb\boldb^*\Qr\boldsymbol{T}^{-*}&=\boldsymbol{0},
\end{align}
which implies
\begin{align}
\Qr\projX\Qr^*\boldA^*\Qr+\boldA\Qr\projX+\boldb\boldb^*\Qr&=\boldsymbol{0}.
\end{align}
Transpose this last expression and use the fact that $\Qr^*\Qr = \boldI_r$ to obtain
\begin{align}
\Qr^*\Qr\projX\Qr^*\boldA^*+\Qr^*\boldA\Qr\projX\Qr^*+\Qr^*\boldb\boldb^*
=
\Qr^*\boldsymbol{R}=\boldsymbol{0},
\end{align}
which is the desired result.
\end{proof}
\begin{remark}
As we have previously noted, in almost every practical situation, one would choose a set of shifts  $\boldsymbol{\sigma}$ which is closed under conjugation. But even for the cases where this assumption on $\poles$ does not hold,  Theorem \ref{sylvester_eq} holds as is, and
 Theorem \ref{thm:lyap} applies with a slight modification.  To wit,
  $\lambda_i =  \bar{\sigma}_i$ for $i=1,\ldots,r$. 
\end{remark}
\section{A numerical study on using the pseudo-$\htwo$ optimal points as the ADI shifts}
Having shown that using the pseudo-$\htwo$ optimal points in the ADI iteration for the Sylvester equation is equivalent to applying RKPM and that the pseudo-$\htwo$ points leads to an orthogonality condition in the case of Lyapunov equation, the natural question to ask is what quality of approximation
the pseudo-$\htwo$ optimal points have  as ADI shifts.  We will briefly investigate this issue in this section. However, we emphasize that 
the purpose of our numerical results is not to advocate employing the pseudo $\mathcal{H}_2$-optimal shifts in the ADI iteration or in the RKPM.  This would be a  costly numerical method for approximating Sylvester equations since obtaining the pseudo $\mathcal{H}_2$-optimal shifts already requires solving several linear systems. Our numerical results are meant to illustrate the unique quality of these shifts compared with other choices of shifts that do not share the ADI-RKPM equivalency property.  

We used three benchmark models in our numerical simulations: The CD Player model with $n=120$,
 the EADY model with $n=598$, and the Rail Model with $n=1357$.  The first two models are described in detail in \cite{chahlaoui2005benchmark} and  the Rail model in \cite{benner2004ens}. For all three models, 
  we compute a rank $r$ approximation to the solution of the  the Lyapunov equation  (\ref{lyapbb}). 
  The exact and approximate solutions are denoted by $\boldsymbol{X}$ and $\boldsymbol{X}_r$, respectively.
  The Rail model has multiple inputs; thus for this model we only use the first column of the input matrix.  
  We use three different approximation methods for each model:
   \begin{itemize}
\item Method 1: The RKPM is applied to the a sequence of shifts that alternates between $0$ and $\infty$.  The resulting subspace is generally referred to as the extended Krylov subspace.  Its application to RKPM was first introduced by Simoncini in \cite{simoncini2008new}.  

\item Method 2:  The RKPM is applied using $r$ pseudo-$\htwo$ optimal shifts; or equivalently $r$ steps 
the ADI iteration is applied using $r$ pseudo-$\htwo$ optimal shifts. 

\item Method 3:  The $r$-steps of ADI iteration are applied where the ADI shifts are chosen
 via Penzl's heuristic method  \cite{penzl200clr}.

\end{itemize}

The quality of the resulting approximations from each method
are compared using the relative error in the $2$-norm,
 i.e. $\displaystyle \frac{\|\boldsymbol{X}-\boldsymbol{X}_r\|_2}{\|\boldsymbol{X}\|_2}$.
Figure \ref{figure:EADY_figure} shows the relative errors  for the EADY model 
as $r$ varies from $1$ to $50$
together with the minimum possible error, i.e. $\frac{\pi_{r+1}}{\pi_1}$ where 
$\pi_i$ is the $i^{\rm th}$ singular value of the true solution $\boldsymbol{X}$. Note that for a given $r$, the pseudo-$\htwo$ optimal shifts perform remarkably well, almost matching the 
best low-rank approximation given by the singular value decomposition. For a selected number of $r$ values, these numbers are also tabulated in Table \ref{table:Eady_table} further illustrating the effectiveness of the pseudo-$\htwo$ points as ADI or RKPM shifts. Similar results for the CD player model are shown in Figure \ref{figure:CD_figure} and in Table \ref{table:CD_table} and for the Rail Model in Figure \ref{figure:rail_figure} and Table \ref{table:rail_table} illustrating that  the pseudo $\mathcal{H}_2$-optimal shifts produce a nearly optimal rank $r$ approximation in several cases. Indeed, 
this phenomenon was recently explained by Breiten and Benner in \cite{breiten_ADI}, where they show that the $\mathcal{H}_2$ optimal shifts are optimal in a special energy norm related to the Lyapunov equation; for details we refer the reader to \cite{breiten_ADI}.

\begin{figure}[hhhh]
\center
\includegraphics[scale=.6]{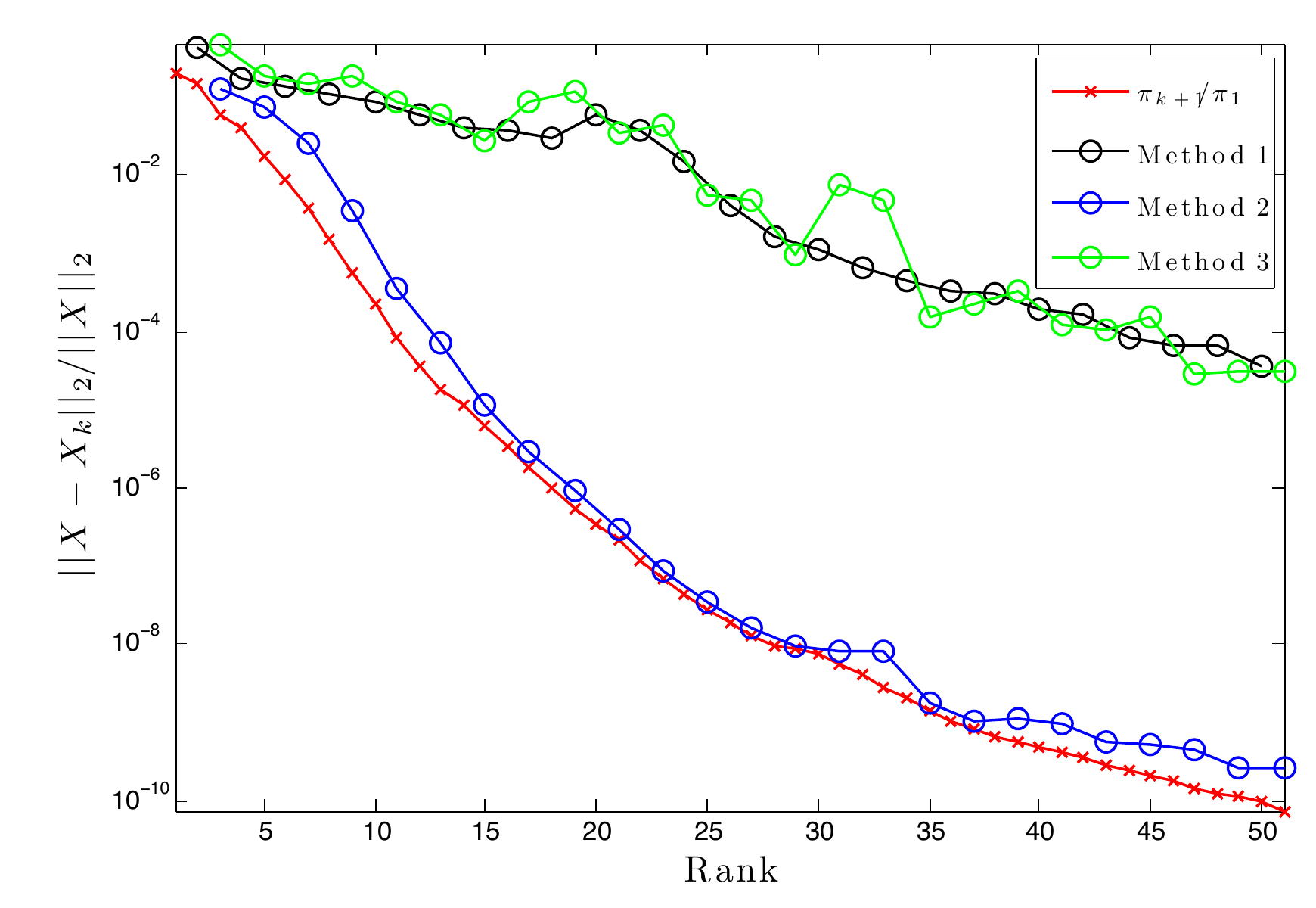}
\caption{Relative error as $r$ varies for the EADY Model }
\label{figure:EADY_figure}
\end{figure}

\begin{table}[hh]
\caption{Comparison of Methods with the EADY model order 598}
\begin{center}
\begin{tabular}{ll|lll} 
& & \multicolumn{3}{c}{$\displaystyle \frac{\|\boldsymbol{X}-\boldsymbol{X}_r\|_2}{\|\boldsymbol{X}\|_2} $}  
 \\
       \hline
  $~r$& ~~~$\frac{\pi_{r+1}}{\pi_1}$ &~Method 1&~Method 2&~Method 3\\ 
  $10$ & $2.31\times 10^{-4}$&$8.42 \times 10^{-2}$ &$1.28\times 10^{-3}$&$1.96 \times 10^{-1}$  \\
  $20$ &$3.38 \times 10^{-7}$&$5.73 \times 10^{-2}$ &$4.99 \times 10^{-7}$&$1.13 \times 10^{-1}$\\ 
  $30$ &$7.63\times 10^{-9}$&$1.09\times10^{-3}$&$8.46\times 10^{-9}$&$8.54\times10^{-2}$ \\
  $40$  &$4.83 \times 10^{-10}$ &$1.93\times 10^{-4}$ &$1.06 \times 10^{-9}$  &$9.70 \times 10^{-3}$
  \end{tabular}
\end{center}
\label{table:Eady_table}
\end{table}

\begin{figure}[hhhh]
\center
\includegraphics[scale=.5]{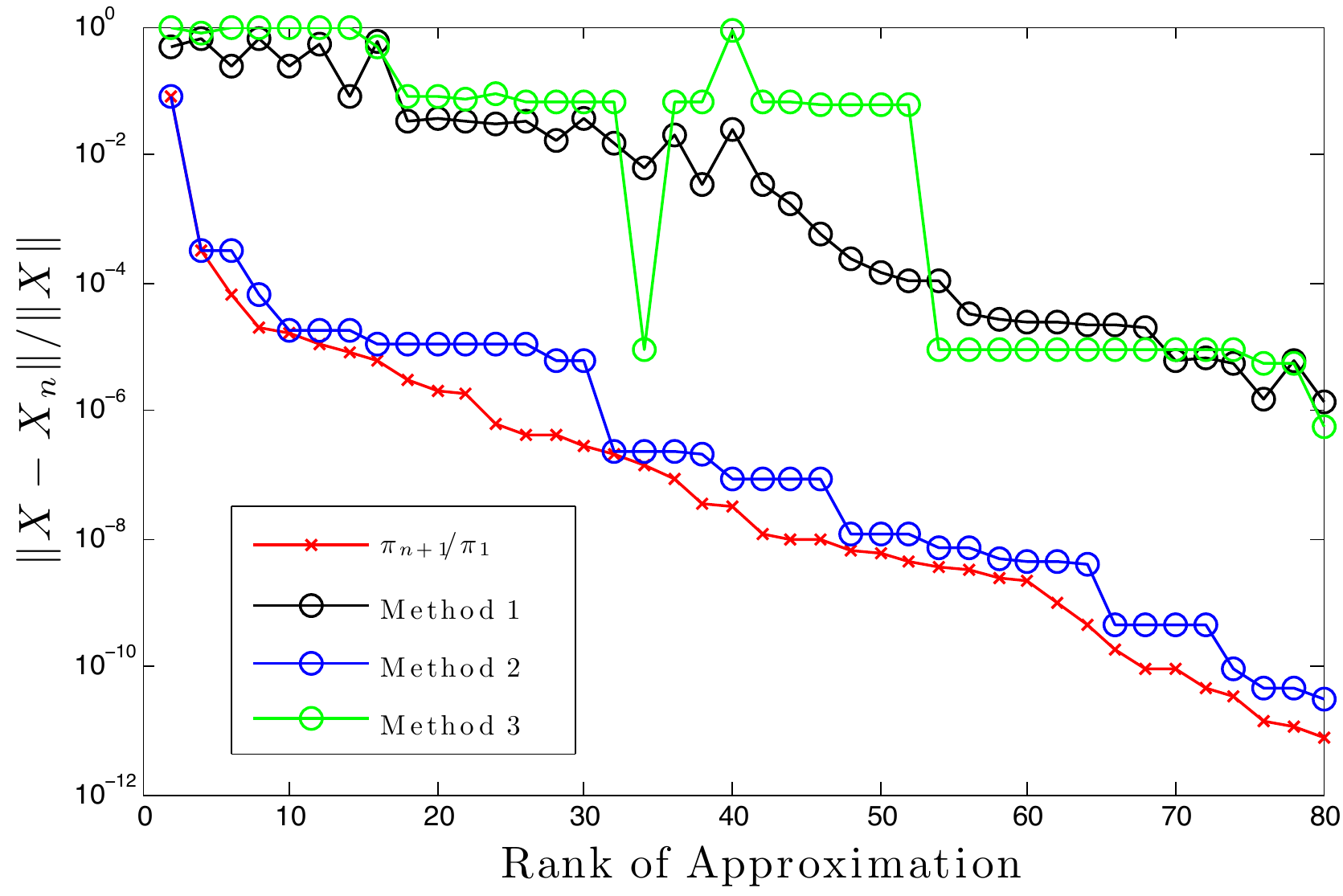}
\caption{Relative error as $r$ varies for for the CD Player model}
\label{figure:CD_figure}
\end{figure}

 \begin{table}[hh]
\caption{Comparison of  Errors for the CD Player model}
\begin{center}
\begin{tabular}{ll|lll} 
& & \multicolumn{3}{c}{$\displaystyle \frac{\|\boldsymbol{X}-\boldsymbol{X}_r\|_2}{\|\boldsymbol{X}\|_2} $}  
 \\
       \hline
  $r$& ~~~$\frac{\pi_{r+1}}{\pi_1}$& ~Method 1& ~Method 2 & ~Method 3\\ 
  $4$ & $3.20 \times 10^{-4}$&6.71$\times 10^{-1}$ &3.20$\times 10^{-4}$& 8.04$\times10^{-1}$\\
 $10$&  $1.68 \times 10^{-5}$ &$2.39 \times 10^{-1} $&$1.73 \times 10^{-5}$&9.92$\times 10^{-1}$\\
 $20$& $1.92\times10^{-6}$&$3.70\times10^{-2}$&$1.10 \times 10^{-5} $&$7.84\times10^{-2}$\\
 $40$& $3.22 \times 10^{-8}$&$1.62\times 10^{-3}$& $8.24 \times 10^{-8}$&$8.56 \times 10^{-2}$
\end{tabular}
\end{center}
\label{table:CD_table}
\end{table}

\begin{figure}[hhhh]
\begin{center}
\includegraphics[scale=.5]{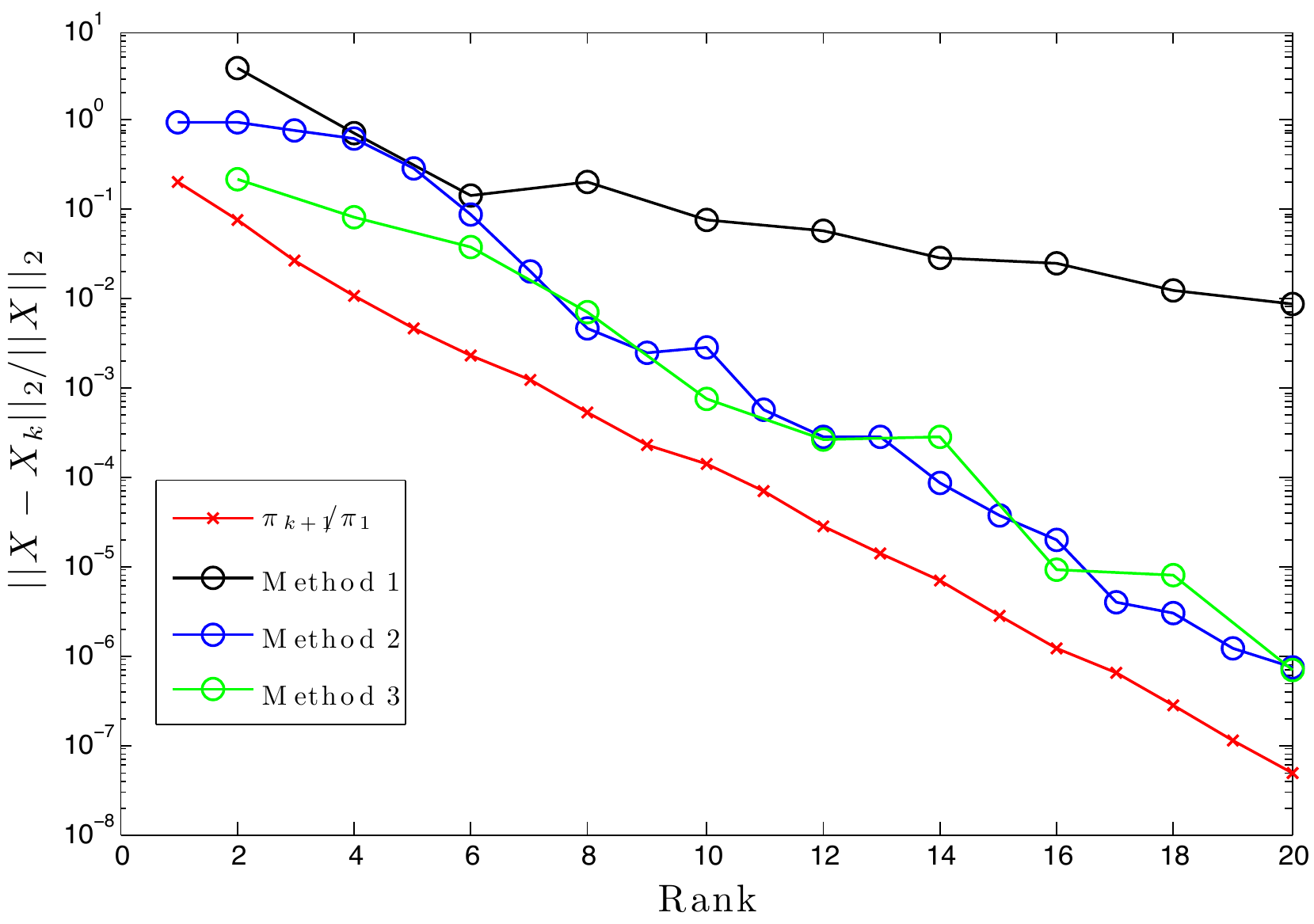}
\end{center}
\caption{Relative error as $r$ varies for for the Rail Model}
\label{figure:rail_figure}
\end{figure}

\begin{table}[hhhh]
\caption{Comparison of  Errors for  the Rail Model}
\begin{center}
\begin{tabular}{ll|lll} 
& & \multicolumn{3}{c}{$\displaystyle \frac{\|\boldsymbol{X}-\boldsymbol{X}_r\|_2}{\|\boldsymbol{X}\|_2} $}  
 \\
       \hline
  $r$ & ~~~$\frac{\pi_{r+1}}{\pi_1}$ &~Method 1&~Method 2&~Method 3\\
  $2$ & $7.47\times10^{-2}$&$3.83 \times 10^0$&$9.47\times10^{-1}$&$2.14\times10^{-1}$\\
 $6$  &$2.31\times 10^{-3}$ &$1.44 \times 10^{-1}$&$8.86 \times 10^{-2}$&3.80$\times 10^{-2}$\\
 $14$ & $6.76\times10^{-6}$&$2.88\times10^{-2}$&$8.43\times 10^{-5}$&$2.90\times10^{-4}$\\
 $20$ & $4.96 \times 10^{-8}$&$8.89\times 10^{-3}$& $7.29 \times 10^{-7}$&$6.85 \times 10^{-7}$
\end{tabular}
\end{center}
\label{table:rail_table}
\end{table}

\section{Acknowledgements}
The authors thank Prof. Christopher Beattie of Virginia Tech. for several fruitful discussions on this subject. 
This work has been supported in part by the NSF Grant DMS-0645347.
\section{Conclusions}
In this paper we presented a new result that solidifies the  connection between the ADI iteration and rational Krylov projection methods for solving large-scale Sylvester equation.  We have shown that for one-sided projections, the two methods are indeed equivalent for a special choice of shifts called pseudo-$\mathcal{H}_2$ optimal shifts, so-called because they partially satisfy first-order necessary conditions for $\mathcal{H}_2$ optimal model reduction.  These shifts are also optimal in the sense that they produce an approximation with a residual orthogonal to the rational Krylov projection subspace in the case of Lyapunov equation. 

\bibliographystyle{plainnat}
\bibliography{references}

\end{document}